\documentclass[11pt]{article}

\usepackage[english]{babel}
\usepackage{amsmath,amsthm,amssymb}
\usepackage{graphicx}
\usepackage{url}
\usepackage{a4wide}
\usepackage{hyperref}
\usepackage{cite}

\theoremstyle{plain}
\newtheorem{thm}{Theorem}[section]
\newtheorem{dd}{Definition}[section]
\newtheorem{lem}{Lemma}[section]
\theoremstyle{definition}
\newtheorem{rem}{Remark}[section]

\DeclareMathOperator{\e}{e}

\newcommand\blfootnote[1]{%
	\begingroup
	\renewcommand\thefootnote{}\footnote{#1}%
	\addtocounter{footnote}{-1}%
	\endgroup
}

\title{Controllability and observability of tempered fractional differential systems\blfootnote{This 
is a preprint version of the paper published open access in 
\emph{Commun. Nonlinear Sci. Numer. Simul.} at \url{https://doi.org/10.1016/j.cnsns.2024.108501}.}}

\author{Ilyasse Lamrani$^{1}$\\
\texttt{i.lamrani@edu.umi.ac.ma} % lamrani.ilyasse.ts@gmail.com
\and
Hanaa Zitane$^{2,3}$\\
\texttt{h.zitane@uae.ac.ma}
\and 
Delfim F. M. Torres$^{3,}$\thanks{Corresponding author.}\\
\texttt{delfim@ua.pt}}
 
\date{$^{1}$TSI Team, Department of Mathematics, Faculty of Sciences,\\ 
Moulay Ismail University, B.P.~11201 Meknes, Morocco\\[0.3cm]
$^{2}$Department of Mathematics, Faculty of Sciences,\\
University of Abdelmalek Essaadi, B.P.~2121 Tetouan, Morocco\\[0.3cm] 
$^{3}$\text{Center for Research and Development in Mathematics and Applications (CIDMA),}  
Department of Mathematics, University of Aveiro, 3810-193 Aveiro, Portugal}

\begin{document}

\maketitle

\begin{abstract}
We study controllability and observability concepts of tempered fractional 
linear systems in the Caputo sense. First, we formulate a solution for the class 
of tempered systems under investigation by means of the Laplace transform method. 
Then, we derive necessary and sufficient conditions for the controllability, 
as well as for the observability, in terms of the Gramian controllability matrix 
and the Gramian observability matrix, respectively. Moreover, we establish the 
Kalman criteria that allows one to check easily the controllability 
and the observability for tempered fractional systems. Applications 
to the fractional Chua's circuit and Chua--Hartley's oscillator models 
are provided to illustrate the theoretical results developed in this manuscript.

\medskip

\medskip

\end{abstract}

% -------------------------------------------------------

\section{Introduction}

Tempered fractional calculus extends the traditional framework of fractional calculus 
by introducing a tempering parameter that controls the decay rate of the memory 
kernel~\cite{Li,Meerschaert}. This extension allows for more flexible modeling 
of systems with long-range memory effects and decay, capturing a wider range of 
behaviors beyond those described by classical fractional operators, 
such as viscoelastic materials, anomalous diffusion, and systems 
where, after a long period, the influence of the past decays at a certain rate.
In recent years, tempered fractional calculus has become a subject 
of investigation due to its applications in stochastic and dynamical systems 
\cite{Meerschaert2,MR1303317,AE:5,Meerschaert3}. For example, it has been used 
to describe and understand turbulence in geophysical flows~\cite{Meerschaert2} 
and L\'{e}vy processes such as Brownian motion~\cite{Meerschaert2,Meerschaert3}. 
Additionally, a truncated L\'{e}vy flight has been investigated to capture the 
natural cutoff in real physical systems \cite{MR1303317}. Furthermore, stochastic 
tempered L\'{e}vy flights have led to significant advancements in mathematical 
research, including solving multi-dimensional partial differential equations 
using both analytical and numerical methods~\cite{Li2}.

In control theory, the controllability and observability concepts are fundamental to 
the study of dynamical systems. Controllability enables one to steer the system's 
state from an initial state to a desired state using control inputs, while 
observability inform us about the ability to reconstruct the system's internal 
state based on its outputs~\cite{curtain2012introduction}. 
Matingon and d'Andréa-Novel were the first to have examined 
the controllability and observability properties of linear fractional 
differential equations in finite dimensional spaces, either in state space form 
or in polynomial representation \cite{MA}. Then, several authors 
have been investigating various results about the controllability 
and observability of linear and nonlinear fractional differential equations, 
see for example~\cite{BGORT,BetDje,Matar,SASA,YJZ} and references cited therein. 
For instance, in \cite{BGORT}, the observability and controllability 
of a fractional time invariant linear system is investigated by means 
of the observability and controllability Grammian matrices that are obtained 
in terms of the Mittag-Leffler matrix function. 
Also, in \cite{YJZ}, necessary and sufficient conditions of controllability 
and observability of fractional continuous time linear systems with regular 
pencils are derived via the construction of Gramian matrices. 
In \cite{balachandran3}, Balachandran et al.\ studied the controllability 
of linear and nonlinear fractional damped dynamical systems involving fractional 
Caputo derivatives of varying orders in finite-dimensional spaces. Their approach 
utilized the Mittag-Leffler matrix function and an iterative technique. 
Furthermore, in \cite{Matar}, Matar 
presented some properties of controllability 
and observability of fractional linear systems using 
conformable fractional derivatives, where he proved that the controllability 
is equivalent to a controllability matrix that has a full rank 
and established a relationship between the controllability 
and the fractional differential Lyapunov equation.

For tempered fractional systems, we have only some results about the stability 
and stabilizability~\cite{Deng,Gassara,Gu,MyID:549}. For instance, in \cite{Deng}, 
the Mittag-Leffler stability of tempered dynamical systems is studied based on the
Lyapunov direct method and the fractional comparison principle. Moreover, 
in \cite{Gassara}, the generalized practical Mittag-Leffler stability 
of a class of tempered fractional nonlinear systems is investigated. In \cite{Gu}, 
the asymptotic and Mittag-Leffler stability of tempered fractional neural networks 
are analyzed by means of the Banach fixed point theorem, while in \cite{MyID:549}, 
a time delay dependent and independent criteria for the finite time 
stability of tempered fractional systems with delay are characterized. 

Motivated by the proceeding, in this work we propose to investigate 
the controllability and observability problems of tempered fractional 
linear systems, which have not yet been tackled. 
The main contributions of this paper can be summarized as follows:
\begin{itemize}
\item An analytical solution to tempered linear systems 
is obtained by means of the Laplace transform method
(see Theorem~\ref{existh}).
\item Necessary and sufficient conditions are established 
for the controllability and observability 
problems by means of Gramian matrices
(see Theorems~\ref{theo2} 
and \ref{Obs1}, respectively).
\item Kalman criteria are derived for the controllability 
and observability of a class of tempered linear systems
(see Theorems~\ref{thm:K:c} 
and \ref{theo4}, respectively).
\item The obtained controllability and observability results are applied 
to the tempered fractional Chua's circuit model 
(see Section~\ref{subsec:Chua:Circuit}) 
and Chua--Hartley's oscillator model
(see Section~\ref{subsec:C:H:oscil}).
\end{itemize}

The remainder of this article is organized as follows. In Section~\ref{Sec:2}, 
we provide the necessary preliminaries, including the definitions of tempered 
fractional derivative and integral in the Caputo sense. Afterwards, we
proceed with Section~\ref{Sec:3}, where we begin by proving an important 
result about the Laplace transform of two functions 
that depend on Mittag-Leffler functions of one and two parameters. Then, 
we establish the analytical solution for tempered fractional linear systems. 
Furthermore, we state our main controllability and observability results. 
Applications that illustrate the obtained results are given 
in Section~\ref{Sec:4}. We end up with Section~\ref{Sec:5} of conclusion
and some possible directions of future research.

% -------------------------------------------------------

\section{Preliminaries}
\label{Sec:2}

In this section, we state some basic definitions and preliminary results 
about tempered fractional operators, which are used in the rest of the paper.

\begin{dd}[See \cite{Li,Meerschaert,Sabatier}]
Let $\alpha>0$, $\rho>0$, and $v$ be an absolutely integrable 
function defined on $[a, b]$, $a, b \in \mathbb{R}$, 
and $a<b$ (if $b=\infty$, then the interval is half-open). 
The tempered fractional integral of function $v$ is defined as follows:
\begin{equation}
\label{TFI}
{ }^T I_a^{\alpha, \rho} v(t)=\frac{1}{\Gamma(\alpha)} 
\int_a^t e^{-\rho(t-s)}(t-s)^{\alpha-1} v(s) \mathrm{d} s,    
\end{equation}
where $\Gamma(\cdot)$ represents the Euler gamma function defined by
$$
\Gamma(z)=\int_0^{\infty} e^{-s} s^{z-1} 
\mathrm{~d} s, \quad z \in \mathbb{C}.
$$
\end{dd}

\begin{dd}[See \cite{Li,Meerschaert}]
\label{TCFD}
Let $\alpha \in(0,1)$ and $\rho>0$. The Caputo tempered fractional 
order derivative of a function $v \in C^1([0, b], \mathbb{R})$ is given by 
\begin{equation}
\label{TFD}
{ }^T D_a^{\alpha, \rho} v(t)=\frac{1}{\Gamma(1-\alpha)} 
\int_a^t e^{-\rho(t-s)}(t-s)^{-\alpha} D^{1, \rho} v(s) \mathrm{d} s    
\end{equation}
with $D^{1, \rho} v(t)=\rho v(t)+d v^{\prime}(t)$, where $d=1$ 
and its dimension equals the dimension of the independent variable $t$.
\end{dd}

\begin{rem}
To ensure the dimensional homogeneity within the combination $\rho v(t)+ dv'(t)$, 
the constant $d = 1$ with a dimension equal to the dimension of the independent 
variable $t$ is added~\cite{MyID:549}.
\end{rem}

\begin{rem}
If we let $\rho=0$ in Definition~\ref{TCFD}, then we retrieve the 
classical fractional derivative in Caputo sense \cite{mittag}.	
\end{rem}

\begin{dd}[See \cite{Mittaglefller}]\label{Mittag}
The Mittag-Leffler functions of a matrix $A$,
of one and two parameters, are defined as
\begin{equation}
\label{eq1}
E_{\alpha}(A)=\sum_{l=0}^{+\infty}\frac{A^{l}}{\Gamma(\alpha l+1)},
\quad Re(\alpha)>0,
\end{equation}
and 
\begin{equation}
\label{eq2}
E_{\alpha,\alpha'}(A)=\sum_{l=0}^{+\infty}
\frac{A^{l}}{\Gamma(\alpha l+ \alpha')},
\quad Re(\alpha)>0,~\alpha'>0,
\end{equation}	
respectively. 
\end{dd}

\begin{lem}[See \cite{Schiff}]
The Laplace transform of a function $\varphi(t)$ 
of a real variable $t \in \mathbb{R}^{+}$ is defined by
$$
(\mathcal{L} \varphi)(s)=\mathcal{L}[\varphi(t)](s)
=\bar{\varphi}(s):=\int_0^{\infty} 
\e^{-s t} \varphi(t) d t, \quad 
s \in \mathbb{C}.
$$
\end{lem}

\begin{lem}[See \cite{Li}]
\label{transf:laplace:tempered}
The Laplace transform of the tempered fractional integral 
\eqref{TFI} and the Caputo derivative \eqref{TFD} are given as
\begin{enumerate}
\item $\mathcal{L}\left[{ }^T I_0^{\alpha, \rho} v(t)\right](s)
=(\rho+s)^{-\alpha} \mathcal{L}[v(t)](s)$;
\item  $\mathcal{L}\left[{ }^T D_0^{\alpha, \rho} v(t)\right](s)
=(s+\rho)^\alpha \mathcal{L}[v(t)](s)-(s+\rho)^{\alpha-1}v(0)$.
\end{enumerate}
\end{lem}

% -------------------------------------------------------

\section{Main results}
\label{Sec:3}

In this paper, we are interested in studying the analytical solution, 
the controllability, as well as the observability, 
of the following linear tempered fractional system:
\begin{equation}
\label{tfs1}
\left\{
\begin{array}{ll}
{ }^T D_0^{\alpha, \rho} y(t)=Ay(t)+Bu(t), 
& t\in [0,T],\\
y(0)=y_{0},  & y_{0} \in \mathbb{R}^{n},
\end{array}
\right.
\end{equation}
where $y:[0, T] \rightarrow \mathbb{R}^n$ is the state, 
${ }^T D_0^{\alpha, \rho} y(t)$ is the Caputo tempered 
fractional order derivative of $y$, $y_{0}$ is the initial state, 
$A \in \mathbb{R}^{n \times n}$ is the state matrix, 
$B \in \mathbb{R}^{n \times m}$ is the control matrix, 
and $u:[0, T] \rightarrow \mathbb{R}^m$ is the input control function.

We first prove the following fundamental 
lemma that enables us to establish Theorem~\ref{existh}. 

\begin{lem}\label{nos}
Let $\alpha\in (0,1)$. Then, 
\begin{enumerate}
\item $\mathcal{L}\left[\e^{-\rho t}E_\alpha\left( At^\alpha\right)\right](s)
=(s+\rho)^{\alpha-1}\left((s+\rho)^\alpha I-A\right)^{-1}$;
\item $\mathcal{L}  \left[\e^{-\rho t}t^{\alpha-1} 
E_{\alpha, \alpha}\left(At^\alpha\right)\right](s)
=\left[(s+\rho)^\alpha I -A\right]^{-1}$.
\end{enumerate}
\end{lem}

\begin{proof}
Using the fact that 
$\displaystyle\mathcal{L}\left[\e^{ct}f(t)\right](s)
=\mathcal{L}[f(t)](s-c)$, 
we have
\begin{equation}
\label{tal12l}	
\mathcal{L}\left[E_\alpha\left( At^\alpha\right)\right]
\left(s+\rho\right)  =\mathcal{L}\left[\e^{-\rho t} 
E_\alpha\left( At^\alpha\right)\right](s)
\end{equation}
and
\begin{equation}
\label{bn101010}	
\mathcal{L}  \left[t^{\alpha-1} E_{\alpha, \alpha}
\left(At^\alpha\right)\right](s+\rho)
=\mathcal{L}  \left[\e^{-\rho t}t^{\alpha-1} 
E_{\alpha, \alpha}\left(At^\alpha\right)\right](s).
\end{equation}
On the other hand, from Definition~\ref{Mittag}, one has
\begin{align*}
\mathcal{L}\left[E_\alpha (A t^\alpha)\right](s+\rho)
&=\mathcal{L}\left[\sum_{k=0}^{\infty}
\frac{(At^{\alpha})^{k}}{\Gamma(k\alpha+1)}\right](s+\rho)\\
&=\sum_{k=0}^{\infty}\frac{A^k}{\Gamma(k\alpha+1)}
\mathcal{L}[t^{\alpha k}](s+\rho)\\
&= \sum_{k=0}^{\infty}\frac{A^k}{\Gamma(k\alpha+1)}
\frac{\Gamma(\alpha k+1)}{(s+\rho)^{\alpha k+1}}.
\end{align*}
Using $\mathcal{L}[t^{\alpha k}](s+\rho)
=\dfrac{\Gamma(\alpha+1)}{(s+\rho)^{\alpha+1}}$ yields
\begin{equation}
\label{12willwr}
\mathcal{L}\left[E_\alpha (A t^\alpha)\right](s+\rho)
= \dfrac{1}{(s+\rho)}\sum_{k=0}^{\infty}
\left(\dfrac{A}{(s+\rho)^\alpha}\right)^{k}.
\end{equation}
Without loss of generality, let us assume that 
$s\geq \|A\|^{\frac{1}{\alpha}}$. 
Then, $\|A\|^{\frac{1}{\alpha}}\leq  (s+\rho)$ and
\begin{align}
\label{3m6}	
\left\|\frac{A}{(s+\rho)^{\alpha}}\right\|<1. 
\end{align}
Applying \eqref{3m6} to \eqref{12willwr}, one gets that
\begin{equation}
\label{e4e4hml}	
\mathcal{L}\left[E_\alpha (A t^\alpha)\right](s+\rho)
=(s+\rho)^{\alpha-1}\left( (s+\rho)^\alpha I-A\right)^{-1}.
\end{equation}
Hence, by combining \eqref{tal12l} and \eqref{e4e4hml}, we obtain that
\begin{equation*}
\mathcal{L}\left[\e^{-\rho t} 
E_\alpha\left( At^\alpha\right)\right](s)
=(s+\rho)^{\alpha-1}\left( (s+\rho)^\alpha I-A\right)^{-1}.
\end{equation*}
This proves the first statement of the lemma. 
For the second statement, we have
\begin{align*}
\mathcal{L}  \left[t^{\alpha-1} 
E_{\alpha, \alpha}\left(At^\alpha\right)\right](s+\rho)
&=\mathcal{L}\left[t^{\alpha-1}\sum_{k=0}^{\infty}
\frac{(At^\alpha)^{k}}{\Gamma(\alpha k+\alpha)}\right](s+\rho)\\
&=\mathcal{L}\left[\sum_{k=0}^{\infty}\frac{t^{\alpha k
+\alpha-1}A^{k}}{\Gamma(\alpha k+\alpha)}\right](s+\rho)\\
&= \sum_{k=0}^{\infty}\frac{A^k}{\Gamma(\alpha k+\alpha)}
\mathcal{L}[t^{\alpha k+\alpha-1}](s+\rho)\\
&= \sum_{k=0}^{\infty}\frac{A^k}{\Gamma(\alpha k+\alpha)}
\frac{\Gamma(\alpha k+\alpha)}{(s+\rho)^{\alpha k+\alpha}}.
\end{align*}
Therefore,
\begin{equation}
\label{Ham1818l}
\mathcal{L}  \left[t^{\alpha-1} E_{\alpha, \alpha}
\left(At^\alpha\right)\right](s+\rho)=\left((s+\rho)^{\alpha}I-A\right)^{-1}.
\end{equation}
By combining \eqref{bn101010} and \eqref{Ham1818l}, it follows that
\begin{equation*}
\mathcal{L}  \left[\e^{-\rho t}t^{\alpha-1} E_{\alpha, \alpha}
\left(At^\alpha\right)\right](s)=\left((s+\rho)^{\alpha}I-A\right)^{-1}.
\end{equation*}
The proof is complete.
\end{proof}

The following theorem characterizes 
the well-posedness of system \eqref{tfs1}.

\begin{thm}
\label{existh}
Let $\rho > 0$ and $\alpha \in (0,1)$. Consider $u(t) \in \mathbb{R}^m$, 
$A \in \mathbb{R}^{n \times n}$, and $B \in \mathbb{R}^{n \times m}$, 
with the property that $\operatorname{det}[(s+\rho)^{\alpha} I - A] \neq 0$ 
for $s \geq \|A\|^{\frac{1}{\alpha}}$. Then, the solution of the linear 
tempered fractional system \eqref{tfs1} is given by
\begin{equation}
y(t) = \e^{-\rho t} 
E_\alpha\left( At^\alpha\right)y_0 
+ \int_{0}^{t} \e^{-\rho (t-\tau)}(t-\tau)^{\alpha-1} 
E_{\alpha,\alpha}\left(A(t-\tau)^{\alpha}\right)Bu(\tau)d\tau.
\end{equation}
\end{thm}

\begin{proof}
Applying  the Laplace transform to both sides 
of the first equation of system \eqref{tfs1}, 
we obtain that
\begin{equation}
\label{eqa}
\mathcal{L}\left[{ }^T D_0^{\alpha, \rho} y(t)\right](s)
=\mathcal{L}\left[Ay(t)+Bu(t)\right](s).
\end{equation}
Using the linearity of the Laplace transform, 
equation~\eqref{eqa} implies that
\begin{equation}
\label{perml}
\begin{split}
\mathcal{L}\left[{ }^T D_0^{\alpha, \rho} y(t)\right](s)
&=\mathcal{L}\left[Ay(t)\right](s)+\mathcal{L}\left[Bu(t)\right](s)\\
&=A\mathcal{L}[y(t)](s)+B\mathcal{L}[u(t)](s).
\end{split}
\end{equation}
Equation \eqref{perml} is obtained due 
to the definition of $\mathcal{L}$. Specifically, $$
\mathcal{L}\left[Ay(t)\right](s)=\displaystyle\int_{0}^{\infty}
\e^{-st}Ay(t)dt=A\displaystyle\int_{0}^{\infty}\e^{-st}y(t)dt
$$
and 
$$
\mathcal{L}\left[Ay(t)\right](s)=\displaystyle\int_{0}^{\infty}
\e^{-st}By(t)dt=B\displaystyle\int_{0}^{\infty}\e^{-st}y(t)dt.
$$
Applying Lemma~\ref{transf:laplace:tempered}, 
we obtain from equation  \eqref{perml} that
\begin{equation*}
(s+\rho)^\alpha \mathcal{L}[y(t)](s)-(s+\rho)^{\alpha-1}y_0
=A\mathcal{L}[y(t)](s)+B\mathcal{L}[u(t)](s).
\end{equation*}
Thus,
\begin{equation*}
\left((s+\rho)^\alpha I-A\right)\mathcal{L}[y(t)](s)
= (s+\rho)^{\alpha-1}y_0+B\mathcal{L}[u(t)](s).
\end{equation*}
Since $\left((s+\rho)^\alpha I-A\right)$ is invertible, we have
\begin{equation}
\label{ygalm}	
\mathcal{L}[y(t)](s)=\left((s+\rho)^\alpha 
I-A\right)^{-1}(s+\rho)^{\alpha-1}y_0
+\left((s+\rho)^\alpha I-A\right)^{-1}B\mathcal{L}[u(t)](s).
\end{equation}
By virtue of Lemma~\ref{nos}, it follows that
\begin{align*}
\mathcal{L}[y(t)](s)
&=\mathcal{L}\left[\e^{-\rho t} 
E_\alpha\left( At^\alpha\right)\right](s)y_0+\mathcal{L}  
\left[\e^{-\rho t}t^{\alpha-1} E_{\alpha, \alpha}
\left(At^\alpha\right)\right](s)B\mathcal{L}[u(t)](s)\\
&=\mathcal{L}\left[\e^{-\rho t} E_\alpha\left( 
At^\alpha\right)\right](s)y_0+\mathcal{L}  \left[
\e^{-\rho t}t^{\alpha-1} E_{\alpha, \alpha}
\left(At^\alpha\right)\ast Bu(t)\right](s),
\end{align*}
where $\ast$ denotes the convolution integral. 
Taking the inverse Laplace transform of the above equation, 
we obtain that 
\begin{equation*}
y(t)=\e^{-\rho t} E_\alpha\left( At^\alpha\right)y_0
+\e^{-\rho t}t^{\alpha-1} 
\ast E_{\alpha, \alpha}\left(At^\alpha\right)Bu(t).
\end{equation*}
Then, from the definition of convolution, we get
\begin{equation*}
y(t)=\e^{-\rho t} E_\alpha\left( At^\alpha\right)y_0
+\int_{0}^{t}\e^{-\rho (t-\tau)}(t-\tau)^{\alpha-1} 
E_{\alpha,\alpha}\left(A(t-\tau)^{\alpha}\right)Bu(\tau)d\tau,
\end{equation*}
which proves the intended result.
\end{proof}

% --------------------------------------

\subsection{Controllability analysis}

In this part, we present some results concerning the controllability 
of system \eqref{tfs1}. Recalling Theorem~\ref{existh}, the solution 
of system \eqref{tfs1} satisfies the following 
variation of constants formula:
\begin{equation}
\label{sol25fps}	
y(t) = \e^{-\rho t} E_\alpha\left( At^\alpha\right)y_0 
+ \int_{0}^{t} \e^{-\rho (t-\tau)}(t-\tau)^{\alpha-1} 
E_{\alpha,\alpha}\left(A(t-\tau)^{\alpha}\right)Bu(\tau)d\tau.
\end{equation}

\begin{dd}
\label{def:cont:orig}
System \eqref{tfs1} is said to be state controllable on $\left[0, T\right]$, 
where $T>0$, if there exists an input signal $u(\cdot): \left[0, T\right] 
\rightarrow \mathbb{R}^m$ such that the associated solution of 
\eqref{tfs1} satisfies $y\left(T\right)=0$.
\end{dd}

Definition~\ref{def:cont:orig} 
concerns controllability at the origin, 
which is equivalent to controllability at any other state 
due to the linearity of the system.

\begin{dd}
\label{Defcon2} 
System \eqref{tfs1} is controllable if for some $T>0$, the controllability 
map $L_{T}:L^{2}\left([0,T]; \mathbb{R}^m\right)\longrightarrow \mathbb{R}^n$ 
has $\mathbb{R}^n$ as range, where
\begin{equation*}
L_{T}u:=\int_{0}^{T}\e^{-\rho (T-\tau)}(T-\tau)^{\alpha-1} 
E_{\alpha,\alpha}\left(A(T-\tau)^{\alpha}\right)Bu(\tau)d\tau.
\end{equation*}
\end{dd}

\begin{rem}
Definition~\ref{Defcon2} is an extension of the alternative 
controllability definition for 
integer order differential linear systems, 
as given in~\cite{curtain2012introduction}.
\end{rem}

A criterion that can be used to construct an open-loop 
control signal consists of introducing the controllability 
Gramian matrix given by
\begin{equation}
\label{GramCon}
W_c\left[0, T\right]:= L_{T}L^{*}_{T} 
=  \int_0^{T} \e^{-2\rho (T-\tau)}(T-\tau)^{2\alpha-2}  
E_{\alpha, \alpha}\left( (T-\tau)^\alpha A\right) B
B^* E_{\alpha, \alpha}\left( (T-\tau)^\alpha A^*\right) 
\mathrm{d}\tau,
\end{equation}
where we denote the matrix transpose by $^*$.

\begin{thm}
\label{theo2}
Under the hypotheses of Theorem~\ref{existh}, system \eqref{tfs1} 
is controllable on $\left[0, T\right]$ if, and only if, 
$W_c\left[0, T\right]$ is non-singular for all $T>0$.
\end{thm}

\begin{proof}
Let $T>0$ and suppose that $W_c\left[0, T\right]$ is 
non-singular. We show that $y(T)=0$. 
Since $W_c\left[0, T\right]$ is non-singular, 
$W_c^{-1}\left[0, T\right]$ is well-defined.
On one hand, replacing $t$ by $T$ in formula \eqref{sol25fps}, we obtain that
\begin{equation}
\label{nstT}	
y(T)=\e^{-\rho T} E_\alpha\left( AT^\alpha\right)y_0
+\int_{0}^{T}\e^{-\rho (T-\tau)}(T-\tau)^{\alpha-1} 
E_{\alpha,\alpha}\left(A(T-\tau)^{\alpha}\right)Bu(\tau)d\tau.
\end{equation}
The control sought is given in the form
\begin{equation} 
\label{controlform}	
u(\tau):=  -(T-\tau)^{\alpha-1}\e^{-\rho(T-\tau)}(B)^* 
E_{\alpha, \alpha}\left( A^*\left(T-\tau\right)^\alpha\right)  
\times W_c^{-1}\left[0, T\right] \e^{-\rho T}  
\times E_\alpha\left( T^\alpha A\right)y_0.
\end{equation}
By replacing \eqref{controlform} in \eqref{nstT}, 
\begin{equation*}
\begin{split}
y(T)
&=\e^{-\rho T} E_\alpha\left( AT^\alpha\right)y_0
+\int_{0}^{T}\e^{-\rho (T-\tau)}(T-\tau)^{\alpha-1} 
E_{\alpha,\alpha}\left(A(T-\tau)^{\alpha}\right)\\
&\qquad B \left\{-(B)^* E_{\alpha, \alpha}^*\left( 
A\left(T-\tau\right)^\alpha\right)  \times W_c^{-1}\left[0, T\right] 
\e^{-\rho T}  \times E_\alpha\left( T^\alpha A\right)y_0 \right\}d\tau\\
&=\e^{-\rho T} E_\alpha\left( AT^\alpha\right)y_0-W_c
\left[0, T\right]W_c^{-1}\left[0, T\right]\e^{-\rho T} 
E_\alpha\left( AT^\alpha\right)y_0\\
&=\e^{-\rho T} E_\alpha\left( AT^\alpha\right)y_0
-\e^{-\rho T} E_\alpha\left( AT^\alpha\right)y_0\\
&=0.
\end{split}
\end{equation*}
Conversely, let us assume that the system is controllable 
and show that the Gramian matrix is invertible. For the sake of argument, 
we suppose that the Gramian matrix is not invertible, that is, 
$\operatorname{det}(A)=0$. Then there exists a nonzero vector $x\in\mathbb{R}^n$ 
with $x^* W_{c}[0,T]x=0$. Thus,
\begin{align*}
&x^*\left\{\int_0^{T} \e^{-2\rho (T-\tau)}(T-\tau)^{2\alpha-2}  
E_{\alpha, \alpha}\left( (T-\tau)^\alpha A\right) B
B^* E_{\alpha, \alpha}^*\left( (T-\tau)^\alpha A\right) \mathrm{d}\tau \right\} x\\
&= \int_0^{T}x^* \e^{-2\rho (T-\tau)}(T-\tau)^{2\alpha-2}  
E_{\alpha, \alpha}\left( (T-\tau)^\alpha A\right) B
B^* E_{\alpha, \alpha}^*\left( (T-\tau)^\alpha A\right) x\mathrm{d}\tau\\
&=0,
\end{align*}
which implies that
\begin{align*}
\int_0^{T} \e^{-2\rho (T-\tau)}(T-\tau)^{2\alpha-2}\|y^*
E_{\alpha,\alpha}\left((T-\tau)^\alpha A\right) B\|^2d\tau=0.
\end{align*}
Since $(T-\tau)^{2\alpha-2}>0$ for all $\tau\in [0, T]$, 
then $\e^{-2\rho (T-\tau)}(T-\tau)^{2\alpha-2}>0$ and
$$
\|y^*E_{\alpha,\alpha}\left((T-\tau)^\alpha A\right)B\|^2=0.
$$
Therefore, 
\begin{equation}
\label{39}	
y^*E_{\alpha,\alpha}\left((T-\tau)^\alpha A\right)B=0.
\end{equation}
Let $y_0=E^{-1}_{\alpha}(T^\alpha A)x$. By the supposition 
of controllability, there exists a control $u$ with the property 
that it moves $y_0$ to the origin. Then, for $y_0=E^{-1}_{\alpha}(T^\alpha A)x$, 
formula \eqref{sol25fps} implies that
\begin{align*}
y(T)
&=\e^{-\rho T} E_\alpha\left( AT^\alpha\right)E^{-1}_{\alpha}(T^\alpha A)x
+\int_{0}^{T}\e^{-\rho (T-\tau)}(T-\tau)^{\alpha-1} 
E_{\alpha,\alpha}\left(A(T-\tau)^{\alpha}\right)Bu(\tau)d\tau\\
&= \e^{-\rho T}x+\int_{0}^{T}\e^{-\rho (T-\tau)}(T-\tau)^{\alpha-1} 
E_{\alpha,\alpha}\left(A(T-\tau)^{\alpha}\right)Bu(\tau)d\tau\\
&=0.
\end{align*}
By taking the scalar product of the equation above with \(x\), we obtain
\begin{align*}
x^*\e^{-\rho T}x+\int_{0}^{T}x^*\e^{-\rho (T-\tau)}(T-\tau)^{\alpha-1} 
E_{\alpha,\alpha}\left(A(T-\tau)^{\alpha}\right)Bu(\tau)d\tau=0.
\end{align*}
Thus,
\begin{align*}
\e^{-\rho T} x^*x+\int_{0}^{T}\e^{-\rho (T-\tau)}(T-\tau)^{\alpha-1}x^* 
E_{\alpha,\alpha}\left(A(T-\tau)^{\alpha}\right)Bu(\tau)d\tau=0.
\end{align*}
It follows from \eqref{39} that
\begin{align*}
\e^{-\rho T} \|x\|^2=0.
\end{align*}
This gives $x=0$, which is absurd with $x$ a nonzero vector. 
Hence, $W_c[0,T]$ is invertible.
\end{proof}

The following result presents a simple algebraic criterion 
to check the controllability of system~\eqref{tfs1}. 
It is an extension of the well-known
Kalman criterion for controllability.

\begin{thm}
\label{thm:K:c}
Under the hypotheses of Theorem~\ref{existh}, 
the tempered fractional system \eqref{tfs1} 
is controllable if, and only if, the Kalman matrix
$$
K=\left[B, A B, A^2 B, \ldots, A^{n-1} B\right]
$$ 
has full rank.
\end{thm}

\begin{proof}
We have
\begin{align*}
y(T)-\e^{-\rho T} E_\alpha\left( AT^\alpha\right)y_0
&=\int_{0}^{T}\e^{-\rho (T-\tau)}(T-\tau)^{\alpha-1} 
E_{\alpha,\alpha}\left(A(T-\tau)^{\alpha}\right)Bu(\tau)d\tau\\
&=\int_{0}^{T}\e^{-\rho (T-\tau)}(T-\tau)^{\alpha-1}
\sum_{k=0}^{\infty}\frac{\left[(T-\tau)^\alpha 
A\right]^k}{\Gamma(\alpha k+\alpha)}Bu(\tau) d\tau.
\end{align*}
By using the uniform convergence, it follows that
\begin{align*}
y(T)-\e^{-\rho T} E_\alpha\left( AT^\alpha\right)y_0
&=\sum_{k=0}^{\infty}\int_{0}^{T}\e^{-\rho (T-\tau)}(T-\tau)^{\alpha-1}
\frac{(T-\tau)^\alpha A^k}{\Gamma(\alpha k+\alpha)}Bu(\tau)d\tau\\
&=\sum_{k=0}^{\infty} A^kB\int_{0}^{T}\e^{-\rho (T-\tau)}(T-\tau)^{\alpha-1}
\frac{(T-\tau)^\alpha}{\Gamma(\alpha k+\alpha)}u(\tau)d\tau\\
&=\lim _{N \rightarrow \infty}\sum_{k=0}^{N} A^kB\int_{0}^{T}
\e^{-\rho (T-\tau)}(T-\tau)^{\alpha-1}
\frac{(T-\tau)^\alpha}{\Gamma(\alpha k+\alpha)}u(\tau)d\tau.
\end{align*}
In the preceding series, each term is expressed as a linear combination 
of the columns of $B$, $A B$, $A^2 B$, \ldots, $A^N B$. Each of these 
matrices can be represented as a linear combination of $B$, $A B$, $A^2 B$, 
\ldots, $A^{n-1} B$. Therefore, the vector
$$
\sum_{k=0}^{N} A^kB\int_{0}^{T}\e^{-\rho (T-\tau)}(T-\tau)^{\alpha-1}
\frac{(T-\tau)^\alpha}{\Gamma(\alpha k+\alpha)}u(\tau)d\tau
$$
is also a linear combination of the columns of 
$B$, $A B$, $A^2 B$, \ldots, $A^{n-1} B$, 
implying that it lies within the range space of the Kalman matrix $K$. 
Consequently, we obtain that $\mathcal{R}(K)=\mathbb{R}^n$.
\end{proof}

\begin{rem}
The condition $\mathcal{R}(K)=\mathbb{R}^n$ is necessary for the 
controllability of the linear fractional system \eqref{tfs1}: 
the Kalman matrix must have full rank, ensuring that we can only 
reach states within the range of the Kalman matrix.
\end{rem}

% ------------------------

\subsection{Observability analysis}
\label{sub:sec:Obs:Anal}

Now, we shall study the observability of the tempered 
fractional linear system~\eqref{tfs1} given by
\begin{equation}
\label{system1}
\left\{
\begin{array}{ll}
{}^T\! D_{0}^{\alpha,\rho}y(t)
=Ay(t)+Bu(t), 
& t\in[0,T],\\
y(0)=y_{0},  & y_{0} \in \mathbb{R}^{n},
\end{array}
\right.
\end{equation}
and augmented with the output function 
\begin{equation}
\label{output}
z(t)=Cy(t)+Du(t),	
\end{equation}
where $z(t)\in\mathbb{R}^p$,  
$C\in \mathbb{R}^{p\times n}$ is the observation matrix,
and $C$ and $D \in \mathbb{R}^{m \times n}$ 
are assumed to be known constant matrices.

\begin{dd}
System~\eqref{system1}--\eqref{output} is said to be state observable 
on $\left[0, t_{f}\right]$, whenever any initial condition 
$y(0)=y_{0} \in \mathbb{R}^{n}$ is distinctively determined via 
the associated system input $u(t)$ and output $y(t)$, 
for any $t \in\left[0, t_{f}\right]$, $t_{f} \in[0, T]$.
\end{dd}

We now derive necessary and sufficient conditions 
for the observability of system~\eqref{system1}--\eqref{output}.

\begin{thm}
\label{Obs1}
System~\eqref{system1}--\eqref{output} is observable on 
$\left[0, t_{f}\right]$ if, and only if, the observability Gramian matrix
$$
W_{o}\left[0, t_{f}\right]:=  \int_{0}^{t_{f}} \e^{-\rho t}
E_{\alpha}^{*}\left(At^{\alpha}\right) C^{*}
C \e^{-\rho t}E_{\alpha}\left(At^{\alpha}\right) \mathrm{d}t
$$
is non-singular for some $t_{f}>0$.
\end{thm}

\begin{proof}
The solution of system~\eqref{system1} is given by
\begin{equation}
\label{Sol}
y(t)=\e^{-\rho t} E_{\alpha}\left(At^{\alpha}\right) y_{0}
+\int_{0}^{t} \e^{-\rho(t-\tau)}(t-\tau)^{\alpha-1}
E_{\alpha, \alpha}\left(A(t-\tau)^{\alpha}\right) 
B u(\tau)  \mathrm{d} \tau.	
\end{equation}
Combining formula~\eqref{Sol} and the output~\eqref{output}, yields 
$$
y(t)=\e^{-\rho t} C E_{\alpha}\left(At^{\alpha}\right) y_{0} 
+\int_{0}^{t} \e^{-\rho(t-\tau)}(t-\tau)^{\alpha-1}
C E_{\alpha, \alpha}\left(A(t-\tau)^{\alpha}\right) 
B u(\tau) \mathrm{d} \tau +D u(t).
$$
Let us consider
$$
\bar{y}(t)= z(t)-\int_{0}^{t} \e^{-\rho(t-\tau)}(t-\tau)^{\alpha-1} 
C E_{\alpha, \alpha}\left(A(t-\tau)^{\alpha}\right) 
B u(\tau)  \mathrm{d} \tau -D u(t).
$$
Thus, one has
$$
\bar{y}(t)=\e^{-\rho t} C E_{\alpha}\left(At^{\alpha}\right) y_{0}.
$$
The observability of the system~\eqref{system1}--\eqref{output} 
is equivalent to the assessment of $y_{0}$ from $\bar{y}(t)$, which, 
in turn, since $\bar{y}(t)$ and $y_{0}$ are arbitrary, 
is equivalent to the evaluation of $y_{0}$ from $y(t)$, 
specified by
$$
y(t)=\e^{-\rho t} C E_{\alpha}\left(At^{\alpha}\right) y_{0},
$$
where $u(t)=0$. The observability Gramian matrix 
$W_{o}^{-1}\left[0, t_{f}\right]$ is well defined provided 
$W_{o}\left[0, t_{f}\right]$ is singular. Then, one gets 
$$
\begin{aligned}
W_{o}^{-1}\left[0, t_{f}\right] \int_{0}^{t_{f}}
&\e^{-\rho t} E_{\alpha}^{*}\left(At^{\alpha}\right) C^{*} y(t) \mathrm{d} t\\
&=W_{o}^{-1}\left[0, t_{f}\right] \int_{0}^{t_{f}} 
\e^{-\rho t}E_{\alpha}^{*}\left(At^{\alpha}\right) C^{*} 
\e^{-\rho t} C E_{\alpha}\left(At^{\alpha}\right) y_{0} \mathrm{d} t\\
&=  W_{o}^{-1}\left[0, t_{f}\right] \int_{0}^{t_{f}} 
\e^{-\rho t} E_{\alpha}^{*}\left(At^{\alpha}\right)C^{*} 
C\e^{-\rho t} E_{\alpha}\left(At^{\alpha}\right) y_{0} \mathrm{d} t\\
&=  W_{o}^{-1}\left[0, t_{f}\right] W_{o}\left[0, t_{f}\right] y_{0}\\
&=y_{0}
\end{aligned}
$$
for arbitrary $y(t)$ and $t_{f}>0$. Therefore,
\begin{equation}
\label{key1}
W_{o}^{-1}\left[0, t_{f}\right] \int_{0}^{t_{f}} 
\e^{-\rho t}E_{\alpha}^{*}\left(At^{\alpha}\right)
C^{*} y(t) \mathrm{d} t=y_{0}.	
\end{equation}
One has that the left side of \eqref{key1} is based on $y(t) \in\left[0, t_{f}\right]$, 
and \eqref{key1} is a linear algebraic equation of $y_{0}$. Using the fact that 
$W_{o}^{-1}\left[0, t_{f}\right]$ exists, it follows that the initial condition 
$y(0)=y_{0}$ is distinctively determined by means of the associated system output 
$y(t)$ for $t \in\left[0, t_{f}\right]$. On the other hand, if $W_{o}\left[0, t_{f}\right]$ 
is singular for some $t_{f}>0$, then there exists a vector $y_{\alpha} \neq 0$ fulfilling
$$
y_{\alpha}^{*} W_{o}\left[0, t_{f}\right] y_{\alpha}=0.
$$
Taking $y_{\alpha}=y_{0}$, one has
$$
y_{\alpha}^{*} \int_{0}^{t_{f}} E_{\alpha}^{*}\left(At^{\alpha}\right) C^{*} 
\e^{-\rho t} \e^{-\rho t} C E_{\alpha}
\left(At^{\alpha}\right)y_{\alpha} \mathrm{d} t=0.
$$
This implies that
$$
\int_{0}^{t_{f}}\|y(t)\|^{2} \mathrm{d} t=0.
$$
Hence,
$$
y(t)=\e^{-\rho t} C E_{\alpha}\left(At^{\alpha}\right) y_{0}
$$
and, since the system is observable, it implies that $y_{0}= 0$, 
which is a contradiction. We conclude that 
$W_{o}\left[0, t_{f}\right]$ is non-singular.
\end{proof} 

We end Section~\ref{sub:sec:Obs:Anal} 
by proving the Kalman condition for the 
observability of system~\eqref{system1}--\eqref{output}.

\begin{thm}
\label{theo4}
System~\eqref{system1}--\eqref{output} is observable 
on $\left[0, t_{f}\right]$ if, and only if,
$$
rank~Q_{o}=rank~\left(
\begin{array}{c}
C \\
C A \\
\vdots \\
C A^{n-1}
\end{array}
\right)=n.
$$	
\end{thm}

\begin{proof}
From the proof of Theorem~\ref{Obs1}, one has
$$
y(t)=\e^{-\rho t} C E_{\alpha}\left(At^{\alpha}\right) y_{0},
$$
where $y_{0}$ is uniquely determined by $y(t)$ if, and only if, 
$C E_{\alpha}\left(At^{\alpha}\right)$ is non-singular. 
Using the Cayley--Hamilton theorem, we obtain that
$$
\begin{aligned}
C E_{\alpha}\left(At^{\alpha}\right)
&=C \sum_{j=0}^{\infty}\frac{t^{\alpha j}}{\Gamma(\alpha j+j)} A^{j}\\
&=C \sum_{j=0}^{n-1} \lambda_{j}t A^{j}\\
&=\sum_{j=0}^{n-1} \lambda_{j}t C A^{j},
\end{aligned}
$$
where $\lambda_{j}$ is a polynomial in $t-a$.
The above formula may be written in matrix form as follows:
$$
C E_{\alpha}\left(At^{\alpha}\right)
=\left(\lambda_{0}t, \lambda_{1}t, \ldots,\lambda_{n-1}t\right)
\left(\begin{array}{c}
C \\
C A \\
\vdots \\
C A^{n-1}
\end{array}\right).
$$
Therefore, $C E_{\alpha}\left(At^{\alpha}\right)$ is non-singular if, 
and only if, rank 
$\left(\begin{array}{c}C \\ C A \\ \vdots \\ C A^{n-1}\end{array}\right)=n$. 
Then, we conclude that system~\eqref{system1}--\eqref{output} is observable 
on $\left[0, t_{f}\right]$ if, and only if, rank $\mathrm{Q}_{o}=n$.
\end{proof}

% -------------------------------------------------------

\section{Applications}
\label{Sec:4}

In this section, we present two examples 
to illustrate our main theoretical results.

% ------------------

\subsection{Example 1: The Fractional-Order Chua's Circuit}
\label{subsec:Chua:Circuit}

Chua's circuit, commonly referred as Chua's oscillator, 
is a basic electronic circuit that demonstrates nonlinear dynamic behaviors, 
such as bifurcation and chaos. This circuit consists of several components, 
including resistors, capacitors, and operational amplifiers. The 
fractional-order Chua's system with incommensurate parameters 
and external disturbance, as described in \cite{Petras}, is defined as follows:
\begin{equation}
\label{Circuitapp}	
\begin{cases}
\displaystyle{ }^T D_0^{\alpha, \rho} y_{1}(t)
=\delta\left(y_{2}(t)-y_{1}(t)-m_1 y_{1}(t)-\frac{1}{2}\left(m_0-m_1\right) 
\times(|y_{1}(t)+1|-|y_{1}(t)-1|)\right), \\
\displaystyle{ }^T D_0^{\alpha, \rho} y_{2}(t)=y_{1}(t)-y_{2}(t)+y_{3}(t), \\
\displaystyle{ }^T D_0^{\alpha, \rho} y_{3}(t)=-\beta y_{2}(t)-\gamma y_{3}(t)+u(t),\\
y_{1}(0)=y_{10},\quad y_{2}(0)=y_{20},\quad y_{3}(0)=y_{30},
\end{cases}
\end{equation}
where $y_{3}(t)$ is the current through the inductance, and $y_{1}(t)$ 
and $y_{2}(t)$ are the voltages across the capacitors. 
In \cite{BukhshY}, when the tempered derivative in \eqref{Circuitapp} 
is replaced with the generalized Caputo proportional fractional-order derivative, 
to consider the system as a linear system, the authors evaluated 
$\frac{1}{2}\left(m_0-m_1\right) \times(|y_{1}(t)+1|-|y_{1}(t)-1|)$ 
for $y_{1}(t)\geq 1$, $-1<y_{1}(t)<1$ and $y_{1}(t)\leq-1$. However, 
this approach might appear unusual, as the term $\frac{1}{2}\left(m_0-m_1\right) 
\times(|y_{1}(t)+1|-|y_{1}(t)-1|)$ is a component of the state solution, 
factored into the determination of the solution, and is no longer a constant 
to be compared with the fixed values $-1$ and $1$. To remove any doubt, 
we linearize the system (as is customary) around the origin, yielding the following system:
\begin{equation}
\label{linCircuitapp}	
\begin{cases}
\displaystyle{ }^T D_0^{\alpha, \rho} y_{1}(t)=\delta\left(y_{2}(t)-y_{1}(t)-m_0 y_{1}(t)\right), \\
\displaystyle{ }^T D_0^{\alpha, \rho} y_{2}(t)=y_{1}(t)-y_{2}(t)+y_{3}(t), \\
\displaystyle{ }^T D_0^{\alpha, \rho} y_{3}(t)=-\beta y_{2}(t)-\gamma y_{3}(t)+u(t),\\
y_{1}(0)=y_{10},\quad y_{2}(0)=y_{20},\quad y_{3}(0)=y_{30}.
\end{cases}
\end{equation}
The linearized system~\eqref{linCircuitapp} 
can be expressed in the abstract form of \eqref{tfs1} as
\begin{equation}
\label{linCircuitapp2}
\begin{cases}
{ }^T D_0^{\alpha, \rho}y(t)=Ay(t)+Bu(t),\\
y(t)=y_0,
\end{cases}
\end{equation}
where
$$
y(t)=
\begin{pmatrix}
y_{1}(t)\\
y_{2}(t)\\
y_{3}(t)
\end{pmatrix},
\quad y_{0}=
\begin{pmatrix}
y_{10}\\
y_{20}\\
y_{30}
\end{pmatrix},
\quad A=
\begin{pmatrix}
-\delta(1+m_0) & \delta & 0\\
1 &-1 & 1\\
0 & -\beta &-\gamma
\end{pmatrix}\: \text{and}\:\: 
B=
\begin{pmatrix}
0\\
0\\
1
\end{pmatrix}.
$$
The controllable matrix has the form
\begin{equation}
\begin{bmatrix}
0 & 0 & \delta \\
0 & 1 & -1-\gamma \\
1 & -\gamma & -\beta+\gamma^2
\end{bmatrix}.
\end{equation}
Moreover, a simple calculation shows that $\det \left[B, A B, A^2 B\right]=\delta\neq 0$, 
which implies that $rank\left[B, A B, A^2 B\right]=3$. Therefore, from Theorem~\ref{theo2}, 
the tempered fractional Chua's circuit is controllable.

Now, we give some numerical simulations for the controllability 
of system~\eqref{linCircuitapp} over the interval $\left[0, 1.5\right]$ 
in the case where $\rho=0.5$ and $ \alpha=0.7$, and by taking the parameters 
$\delta=2$, $\beta=0.5$, $\gamma=-1$, and $m_{0}=3$.
According to system~\eqref{linCircuitapp2}, one has 
$$
A=\begin{pmatrix}
-8 & 2 & 0\\
1 &-1 & 1\\
0 & -0.5 &1
\end{pmatrix}.
$$
By virtue of the controllability Gramian matrix formula~\eqref{GramCon}, we have 
\begin{equation}
\begin{aligned}
W_c\left[0, 1.5\right]= \int_0^{1.5} \e^{-(1.5-\tau)}(1.5-\tau)^{-0.6}  
E_{0.7, 0.7}\left( (1.5-\tau)^{0.7} A\right) 
B B^* E_{\alpha, \alpha}\left( (1.5-\tau)^{0.7} A^*\right) \mathrm{d}\tau.
\end{aligned}
\end{equation}
Using the Mittag-Lefller function of two parameters~\eqref{eq2}, we obtain that
$$
W_c\left[0, 1.5\right]
=
\begin{pmatrix}
17.4120 & 9.7422 & 10.0512\\
9.7422 & 2.8135 & 24.3021\\
10.0512 & 24.3021 & 25.3369\\
\end{pmatrix},
$$
which is non-singular for $T=1.5$. Then, from Theorem~\ref{theo2}, we conclude 
that system~\eqref{linCircuitapp} with $\rho=0.5$ is controllable on $[0 ~1.5]$.
We shall use the control $u(t)$ defined by formula~\eqref{controlform} as follows:
\begin{multline}
\label{Control1}
u(t) = -(1.5 - t)^{-0.3}\e^{-0.5(1.5 - t)}(B)^* E_{0.7, 0.7}
\left( A^*\left(1.5 - t\right)^{0.7}\right)\\
\times W_c^{-1}\left[0, 1.5\right] \e^{-0.75} 
\times E_{0.7}\left( 1.5^{0.7} A\right)y_0.
\end{multline}
This control can steer the system~\eqref{linCircuitapp} from the initial state 
$y(0) = [2; 5; 3]$ to the desired state $y(1.5) = [1; -3.5; 3.5]$.

Figure~\ref{Figure1} presents the trajectories of system~\eqref{linCircuitapp} 
without the control $u(t)$. We observe that there is no trajectory connecting 
the initial state $[2; 5; 3]$ to the final desired state $[1; -3.5; 3.5]$. In 
Figure~\ref{Figure2}, we examine the behavior of system~\eqref{linCircuitapp} 
using the steering control $u(t)$. It shows that the state of system~\eqref{linCircuitapp} 
moves from its initial state to the desired final state $[1; -3.5; 3.5]$ over the interval 
$[0,~1.5]$. The evolution of the steering control function $u(t)$ 
is illustrated in Figure~\ref{Figure3}.
% --------------------------------------------------------
\begin{figure}[ht!]
\begin{center}
\includegraphics[scale=0.8]{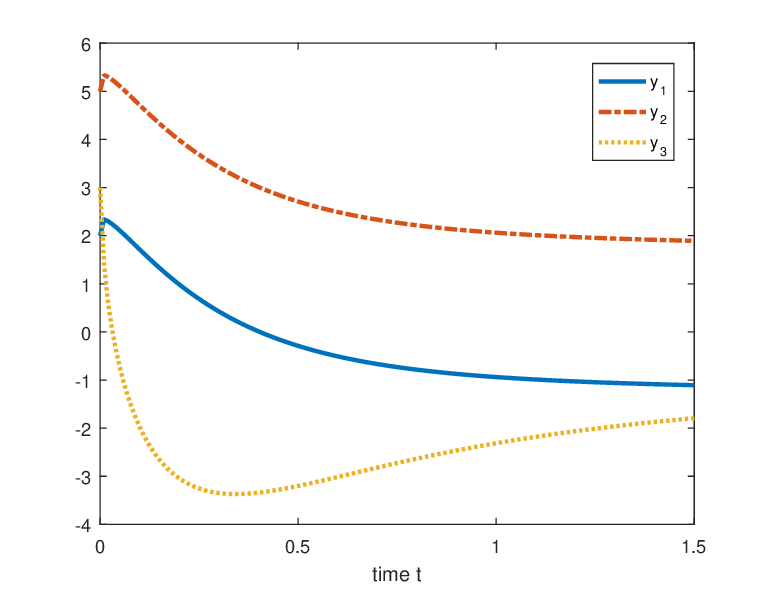}
\caption{The system~\eqref{linCircuitapp} evolution on the interval $[0 ~1.5]$ 
with $\rho=0.5$ and $u(t)=0$.}\label{Figure1}
\end{center}
\end{figure}
% --------------------------------------------------------
\begin{figure}[ht!]
\begin{center}
\includegraphics[scale=0.8]{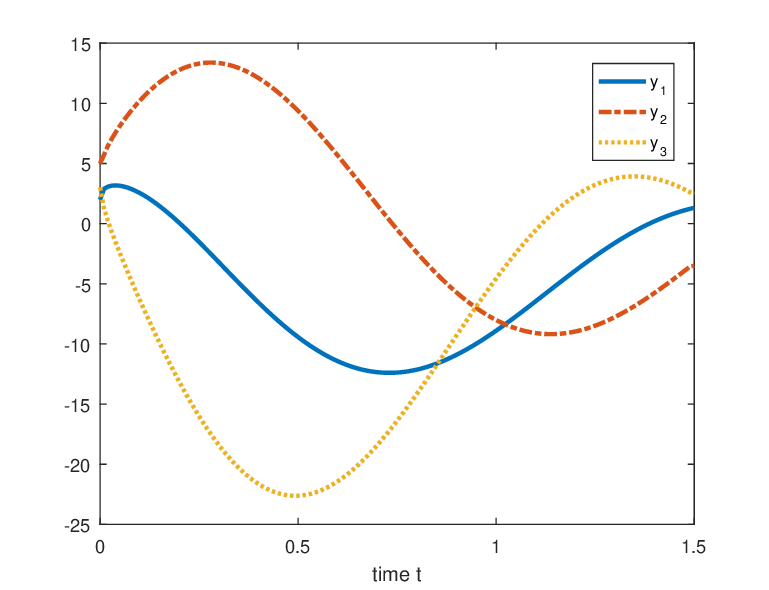}
\caption{The evolution of the controlled system~\eqref{linCircuitapp} 
with $\rho=0.5$ on the interval $[0~ 1.5]$.}\label{Figure2}
\end{center}
\end{figure}
% --------------------------------------------------------
\begin{figure}[ht!]
\begin{center}
\includegraphics[scale=0.8]{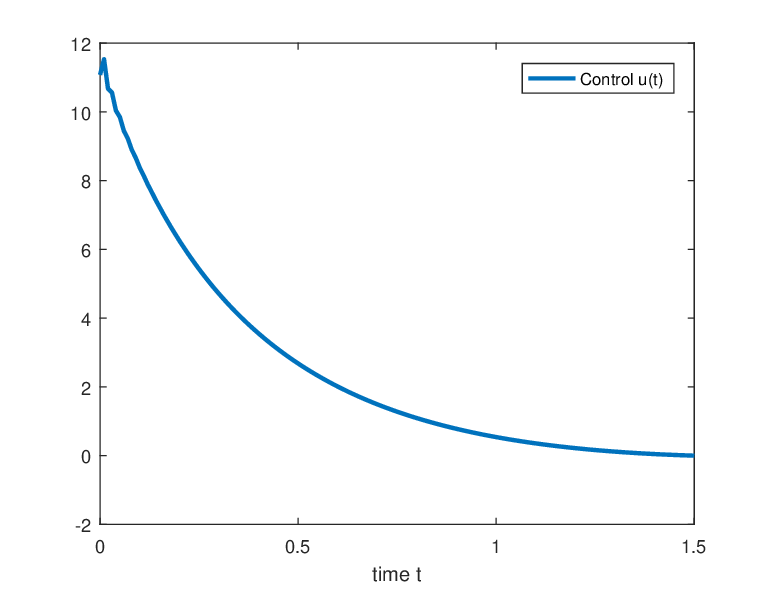}
\caption{The steering control~\eqref{Control1} 
on $[0~ 1.5]$ in the case $\rho=0.5$.}\label{Figure3}
\end{center}
\end{figure}
% --------------------------------------------------------

Now, we present numerical simulations for the controllability 
of system~\eqref{linCircuitapp} with $\rho = 1$ on the interval 
$\left[0, 1.5\right]$. We again take the parameter values 
$\delta = 2$, $\beta = 0.5$, $\gamma = -1$, and $m_{0} = 3$.

From the controllability Gramian matrix formula~\eqref{GramCon}, we have
\begin{equation*}
\begin{aligned}
W_c\left[0, 1.5\right]=\int_0^{1.5} \e^{-2 (1.5-\tau)}(1.5-\tau)^{0.6}  
E_{0.7, 0.7}\left( (1.5-\tau)^{0.7} A\right) 
B B^* E_{0.7, 0.7}\left( (1.5-\tau)^{0.7} A^*\right) \mathrm{d}\tau.
\end{aligned}
\end{equation*}	
Using the Mittag-Lefller function of two parameters~\eqref{eq2} 
and performing some numerical calculations, we get
$$
W_c\left[0, 1.5\right]
=
\begin{pmatrix}
18.2342 & 10.2913 & 10.4490\\
10.2913& 3.7710 & 24.1524 \\
10.4490 & 24.1524  & 27.5196\\
\end{pmatrix},
$$
which is non-singular for $T=1.5$. Then, by virtue of Theorem~\ref{theo2}, 
we deduce that system~\eqref{linCircuitapp} with $\rho=1$ is controllable 
on $[0 ~1.5]$. We use the control $u(t)$ defined by formula~\eqref{controlform} 
as follows:
\begin{equation}
\label{Control2}
u(t)=  -(1.5-t)^{-0.3}\e^{-1.5+t}(B)^* E_{0.7, 0.7}\left( A^*\left(1.5-t\right)^{0.7}\right)  
\times W_c^{-1}\left[0, 1.5\right] \e^{-T}  \times E_{0.7}\left( 1.5^{0.7} A\right)y_0
\end{equation}
that allows us to steer the system~\eqref{linCircuitapp} from
the	initial state  $y(0)=[4;2;-3]$ 
to the final state $y(1.5)=[5; 9; 15]$.

In Figure~\ref{Figure4}, we examine the solution of system~\eqref{linCircuitapp} 
without the control $u(t)$. We observe that there is no trajectory connecting 
the initial state $[4; 2; -3]$ and the final state $[5; 9; 15]$. 
In Figure~\ref{Figure5}, we illustrate the behavior of system~\eqref{linCircuitapp} 
using the steering control $u(t)$. It shows that the state of system~\eqref{linCircuitapp} 
moves from its initial state to the final state $[5; 9; 15]$ over the interval $[0,~1.5]$. 
The steering control function $u(t)$ is shown in Figure~\ref{Figure6}. 
	
% --------------------------------------------------------
\begin{figure}[ht!]
\begin{center}
\includegraphics[scale=0.8]{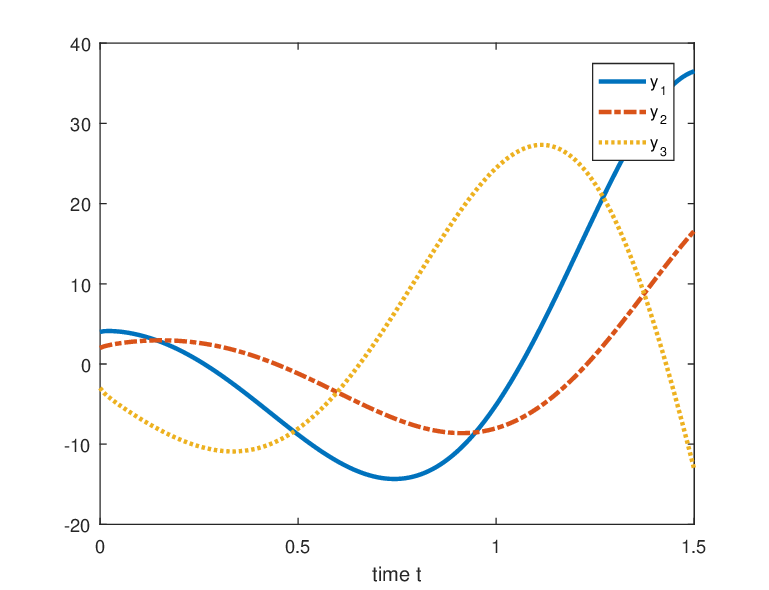}
\caption{The system~\eqref{linCircuitapp} evolution 
on the interval $[0 ~1.5]$ with $\rho=1$ 
and $u(t)=0$.}\label{Figure4}
\end{center}
\end{figure}
% --------------------------------------------------------
\begin{figure}[ht!]
\begin{center}
\includegraphics[scale=0.8]{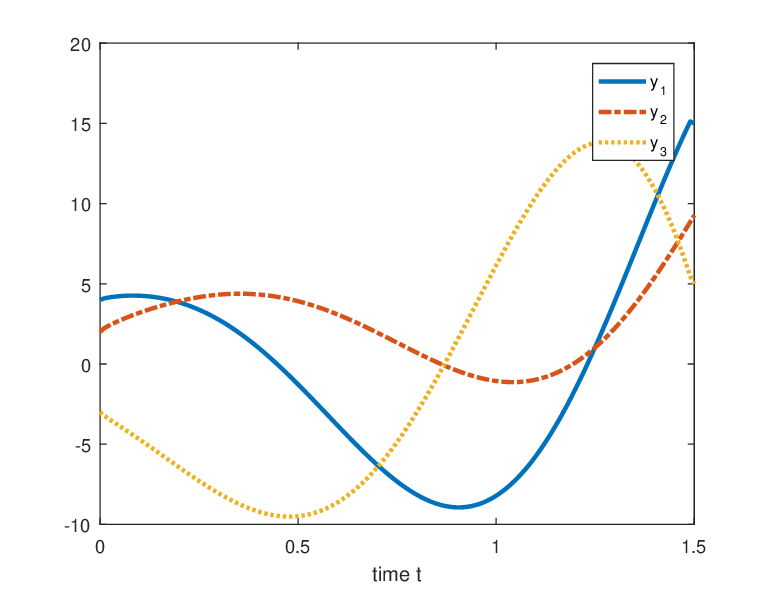}
\caption{The evolution of the controlled 
system~\eqref{linCircuitapp} with $\rho=1$ on the interval 
$[0~ 1.5]$.}\label{Figure5}
\end{center}
\end{figure}
% --------------------------------------------------------
\begin{figure}[ht!]
\begin{center}
\includegraphics[scale=0.8]{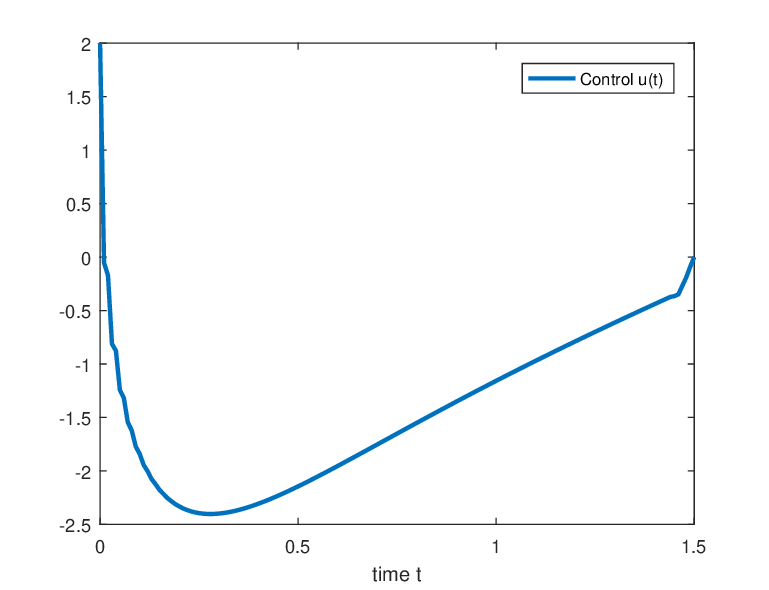}
\caption{The steering control~\eqref{Control2} 
on $[0~ 1.5]$ in the case  $\rho=1$.}\label{Figure6}
\end{center}
\end{figure}

% -----------------------------------

\subsection{Example 2: Chua--Hartley's oscillator}
\label{subsec:C:H:oscil}

The Chua--Hartley system differs from the conventional Chua system 
in such a way that it substitutes the piecewise-linear nonlinearity 
with an appropriate cubic nonlinearity, resulting in a very similar behavior. 
Here, we replace the integer derivatives on the left side of the  Chua--Hartley 
differential system~\cite{Hartley} by tempered fractional derivatives, as follows:
\begin{equation}
\label{oscia2}	
\begin{cases}
{ }_0 { }^T D_0^{\alpha, \rho} y_{1}(t)
=\displaystyle\delta\left(y_{2}(t)+\frac{y_{1}(t)-2 y_{1}^3(t)}{7}\right),\\
{ }^T D_0^{\alpha, \rho} y_{2}(t)=\displaystyle y_{1}(t)-y_{2}(t)+y_{3}(t)+u(t),\\
{ }^T D_0^{\alpha, \rho} y_{3}(t)=-\beta y_{2}(t)=\displaystyle-\frac{100}{7} y_{2}(t).
\end{cases}
\end{equation}
Also, we consider that the system is augmented with the output functions
\begin{equation}
\begin{cases}
z_1(t)=y_{1}(t)+y_{3}(t),\\
z_2(t)=2y_{1}(t)-y_{2}(t).
\end{cases}\label{augobs}
\end{equation}
Linearizing the system~\eqref{oscia2} around the origin, 
one obtains that
\begin{equation}
\begin{cases}
{ }_0 { }^T D_0^{\alpha, \rho} y_{1}(t)
=\displaystyle\delta y_{2}(t)+\frac{\delta}{7}y_{1}(t), \\
{ }^T D_0^{\alpha, \rho} y_{2}(t)=\displaystyle y_{1}(t)-y_{2}(t)+y_{3}(t)+u(t), \\
{ }^T D_0^{\alpha, \rho} y_{3}(t)=\displaystyle-\frac{100}{7} y_{2}(t).
\end{cases}\label{oscia2l}
\end{equation}
According to \eqref{system1}--\eqref{output}, the linearized 
system~\eqref{oscia2l}--\eqref{augobs} can be expressed 
in the following abstract form:
\begin{equation}
\begin{cases}
{ }^T D_0^{\alpha, \rho}y(t)=Ay(t)+Bu(t),\\
y(t)=y_0,
\end{cases}
\end{equation}
and
\begin{equation}
z(t)=Cy(t), 
\end{equation}
where
$$
y(t)
=
\begin{pmatrix}
y_{1}(t)\\
y_{2}(t)\\
y_{3}(t)
\end{pmatrix},\ 
A=
\begin{pmatrix}
\displaystyle\frac{\delta}{7} & \delta & 0\\
1 &-1 & 1\\
0 & -\frac{100}{7} &0
\end{pmatrix},\: \: 
B=\begin{pmatrix}
0\\
1\\
0
\end{pmatrix}, 
z(t)
=
\begin{pmatrix}
z_1(t)\\
z_2(t)
\end{pmatrix}\
\text{and}\ 
C
=
\begin{pmatrix}
1 &0 & 1\\
2& -1& 0
\end{pmatrix}. 
$$
Simple calculations show that
\begin{align*}
rank~Q_{o}
&=rank~\left(\begin{array}{c}
C \\
C A \\
C A^{2}
\end{array}\right)\\
&=rank~\begin{pmatrix}
1  & 0& 1\\
2 & -1 & 0\\
\displaystyle\frac{\delta}{7} & \delta-\frac{100}{7}& 0\\
\displaystyle\frac{2\delta}{7}-1 &2\delta+1 & -1\\
\displaystyle \frac{2\delta^2}{49}+2\delta-\frac{100}{7}
&\displaystyle \frac{2\delta^2}{7}-2\delta+\frac{100}{7}
&\displaystyle 2\delta-\frac{100}{7}\\
\displaystyle \frac{2\delta^2}{49}+\frac{10\delta}{7}+1
&\displaystyle \frac{2\delta^2}{7}-3\delta+\frac{93}{7}
&\displaystyle 2\delta+1
\end{pmatrix}\\
=3.
\end{align*}
Therefore, from Theorem~\ref{theo4}, the tempered 
linearized Chua--Hartley oscillator is observable.

% -----------------------------------

\section{Conclusion}
\label{Sec:5}

In this paper, we studied the well-posedness, controllability, and observability 
of tempered fractional differential systems. The main contribution is the proof 
of necessary and sufficient Gramian matrix criteria and rank conditions 
for both the controllability and observability of linear tempered fractional 
differential systems. To illustrate the theoretical results, we provided examples 
using fractional Chua's circuit and Chua--Hartley oscillator models. 

Controllability results for nonlinear fractional-order dynamical systems
can be found in \cite{balachandran2,balachandran1}. In the future, 
we plan to investigate the controllability and observability problems
for nonlinear tempered fractional differential systems. 

% -----------------------------------

\section*{Funding}

Zitane and Torres were supported by The Center for Research and
Development in Mathematics and Applications (CIDMA)
through the Portuguese Foundation for Science and Technology 
(FCT), projects UIDB/04106/2020 (\url{https://doi.org/10.54499/UIDB/04106/2020})
and UIDP/04106/2020 (\url{https://doi.org/10.54499/UIDP/04106/2020}).

% -----------------------------------

\section*{Acknowledgements}

The authors are grateful to constructive and helpful editorial 
and reviewers comments, that helped them to substantially 
improve the first submitted version.

% -----------------------------------

% -----------------------------------

\end{document}